\documentclass[12pt,verbatim]{amsart}
\usepackage{amsmath,amssymb,amsfonts,amsthm}
\usepackage{indentfirst}
\usepackage[top=2.5cm,bottom=2cm,left=2.5cm,right=2cm]{geometry}
\usepackage{amsmath}
\usepackage{amsthm}
\usepackage{amsfonts}
\usepackage{bbm}
\usepackage{CJK}
\usepackage{fancyhdr}
\usepackage{graphicx}
\usepackage{geometry}
\usepackage{psfrag}
\usepackage{amsfonts,amsmath,amsthm, amssymb}
\usepackage{latexsym, euscript, epic, eepic}
\usepackage{time}
\usepackage{txfonts}
\usepackage{colortbl}
\usepackage{stmaryrd}
\usepackage{mathrsfs}
\usepackage{txfonts}
\usepackage{amsfonts}
\usepackage{color}
\usepackage{lineno}
\usepackage[square, comma, sort&compress, numbers]{natbib}
\usepackage{indentfirst,latexsym,bm}

\newtheorem{theorem}{Theorem}

\newtheorem{corollary}{Corollary}
\newtheorem{definition}{Definition}[section]
\newtheorem{lemma}{Lemma}

\numberwithin{equation}{section}

\title{A new metric associated with the domain boundary}

\author{Xingchen Song} 
\address{Xingchen Song: School of Science, Zhejiang Sci-Tech University, Hangzhou 310018, China}
\email{202120102051@mails.zstu.edu.cn}

\author{Gendi Wang*} 
\address{*Gendi Wang ( Corresponding author ): School of Science, Zhejiang Sci-Tech University, Hangzhou 310018, China}
\email{gendi.wang@zstu.edu.cn}

\begin{document}  

\newcounter{minutes}\setcounter{minutes}{\time}
\divide\time by 60
\newcounter{hours}\setcounter{hours}{\time}
\multiply\time by 60 \addtocounter{minutes}{-\time}
\def\thefootnote{}
\footnotetext{ {\tiny File:~\jobname.tex,
          printed: \number\year-\number\month-\number\day,
          \thehours.\ifnum\theminutes<10{0}\fi\theminutes }}
\makeatletter\def\thefootnote{\@arabic\c@footnote}\makeatother

\maketitle

\begin{abstract}
In this paper, we introduce a new metric $\tilde{c}$ which is associated with the domain boundary for a Ptolemy space $(X,d)$.
Moreover, we study the inclusion relation of the $\tilde{c}$ metric balls and some related hyperbolic type metric balls in subdomains of $\mathbb{R}^n$.
In addition,  we study distortion properties of M\"obius transformations with respect to the $\tilde{c}$ metric in the unit ball
and the quasiconformality of bilipschitz mappings in $\tilde{c}$ metric.
\end{abstract}

{\small \sc Keywords.} {$\tilde{c}$ metric; metric ball; ball inclusion; M\"obius transformation; bilipschitz map}

{\small \sc 2020 Mathematics Subject Classification.} {30F45 (51E05)}

\section{Introduction}

Hyperbolic metrics play an important role in driving the development of geometric function theory.
However, hyperbolic metrics have limitations in high-dimensional spaces.
As a result, hyperbolic type metrics are introduced as alternatives to hyperbolic metrics.
The most familiar hyperbolic type metrics include the quasihyperbolic metric, the distance ratio metric, the Cassinian metric, and  the triangular ratio metric, etc.
In hyperbolic type metrics, various metric function families defined in relation to the domain boundary are introduced,
and their geometric and analytical properties are studied \cite{jx,wyp,xxx,jw10,jw11,jx10,jw}.

In 2002, H{\"a}st{\"o} \cite{jw7} introduced the generalized relative metric named as the $M$-relative metric.
Let $D\subsetneq \mathbb{R}^n$, and $M(x,y)$ be continuous on $(0,\infty)\times (0,\infty)$, and $\rho_M(x,y)=\frac{|x-y|}{M(|x|,|y|)}$ be a metric.
Then for all $x,y\in D$,
\begin{equation*}
\rho_{M,D}(x,y)=\underset{p\in\partial D}{\sup}\frac{|x-y|}{M(|x-p|,|y-p|)}
\end{equation*}
is also a metric.
When $M(x,y)=x+y$, the corresponding metric is the triangular ratio metric
\begin{equation*}
s_D(x,y)=\underset{p\in\partial D}{\sup} \frac{|x-y|}{|x-p|+ |y-p|}.
\end{equation*}
The triangular ratio metric has been recently investigated in \cite{jw4,jw9,jw10,jw11}.

In this paper, we define a new hyperbolic type metric $\tilde{c}_D(x,y)$ associated with the domain boundary $\partial D$:
let $(X,d)$ be a Ptolemy space, $D\subsetneq X$ be an open set with $\partial D \ne\emptyset$,
then for all $x,y\in D$,
\begin{equation*}
\tilde{c}_D(x,y)=\underset{p\in\partial D}{\sup} \frac{d(x,y)}{\max\{d(x,p), d(y,p)\}}\,.
\end{equation*}
For the proof see Theorem \ref{mainthm} in Section 3.

We investigate the inclusion relation of the $\tilde{c}$ metric balls in subdomains of $\mathbb{R}^n$.
Additionally, we study the distortion properties of the $\tilde{c}$ metric under M\"obius transformations on the unit ball.
We have known that mappings satisfying the bilipschitz condition with respect to the triangular ratio metric is quasiconformal.
Due to certain similarities between the triangular ratio metric and the $\tilde{c}$ metric,
we prove the quasiconformality of bilipschitz mappings in $\tilde{c}$ metric as well.

\bigskip
\section{Hyperbolic Type Metrics}

In this section, we review the definitions of some related hyperbolic type metrics.
Let $X$ be a metric space and $D\subsetneq X$ be an open set with $\partial D \ne\emptyset$.
Let $d(x)=d(x,\partial D)$ be the distance from the point $x$ to the boundary of $D$,
and $ \min\{d(x,p),d(y,p)\}=d(x,p)\land d(y,p), \, \max\{d(x,p),d(y,p)\}=d(x,p)\vee d(y,p),\, {d_{min}}=d(x)\land d(y),\, d_{max}=d(x)\vee d(y).$

\subsection{The hyperbolic metric}

The hyperbolic metrics $\rho_{\mathbb{H}^n}$ and $\rho_{\mathbb{B}^n}$
of the upper half space $\mathbb{H}^n=\{(x_1,...,x_n)\in \mathbb{R}^n:x_n>0\}$ and of the unit ball $ \mathbb{B}^n=\{z\in \mathbb{R}^n:|z|<1\}$ are, respectively, defined
as follows \cite{jx4}: for all $x, y \in {\mathbb{H}^n}$
\begin{equation*}
\cosh\frac{\rho_{\mathbb{H}^n}(x,y)}{2}=1+\frac{{|x-y|}^2}{2x_ny_n},
\end{equation*}
and for all $x, y \in {\mathbb{B}^n}$
\begin{equation*}
\sinh\frac{\rho_{\mathbb{B}^n}(x,y)}{2}=\frac{|x-y|}{\sqrt{(1-{|x|}^2)(1-{|y|}^2)}}.
\end{equation*}

\medskip
\subsection{The distance ratio metric}

Gehring and Osgood first introduced the original definition of distance ratio metric \cite{gehring1979uniform}.
The following form was proposed by Vourinen in \cite{jw20}.
For all $x,y\in D\subset\mathbb{R}^n$, the distance ratio metric is defined as
\begin{equation*}
j_D(x,y)=\log\left(1+ \frac{|x-y|}{d_{min}}\right).
\end{equation*}

It follows from \cite{jx9} and \cite{jw20} that for all $x,y\in D\in \{\mathbb{B}^n, \mathbb{H}^n\}$,
\begin{equation}\label{jrhoB}
j_D(x,y)\leq \rho_D(x,y)\leq 2j_D(x,y).
\end{equation}

\medskip
\subsection{The quasihyperbolic metric}

For all $x,y\in D\subsetneq\mathbb{R}^n$, the quasihyperbolic metric is defined as \cite{jw6}
\begin{equation*}
k_D(x,y)=\underset{\gamma\in\Gamma(x,y)}{\inf} \int_{\gamma}{\frac{|dz|}{d(z)}},
\end{equation*}
where $\Gamma(x,y)$ represents the set of rectifiable arcs in $D$ connecting points $x$ and $y$.

It follows from \cite{jw6} that for all $x, y\in D\subsetneq\mathbb{R}^n$,
\begin{equation}\label{jk}
j_D(x,y)\leq k_D(x,y).
\end{equation}

\medskip
\subsection{\boldmath{The Barrlund metric}}

For $q\geq 1$ and all $x,y\in D\subset\mathbb{R}^n$, the Barrlund metric is defined as \cite{bDp}
\begin{equation*}
b_{D,q}(x,y)=\underset{p\in \partial D}\sup\frac{|x-y|}{\left(|x-p|^q+|y-p|^q\right)^{1/q}}.
\end{equation*}
When $q=1$, the Barrlund metric is a special case known as the triangular ratio metric
\begin{equation*}
b_{D,1}(x,y)=s_D(x,y)=\underset{p\in\partial D}{\sup} \frac{|x-y|}{|x-p|+ |y-p|}.
\end{equation*}

\medskip
\subsection{The Cassinian metric}

Ibragimov introduced the Cassinian metric and studied its geometric properties \cite{jx10}.
For all $x,y\in D\subset\overline{\mathbb{R}^n}$, the Cassinian metric is defined as
\begin{equation*}
c_D(x,y)=\underset{p\in\partial D}{\sup} \frac{|x-y|}{|x-p||y-p|}.
\end{equation*}

It follows from \cite{jw22} that for all $x, y\in D\subsetneq\mathbb{R}^n$,
\begin{equation*}
\sinh\frac{\rho_{\mathbb{B}^n}(x,y)}{2}\leq c_{\mathbb{B}^n}(x,y).
\end{equation*}

\medskip
\subsection{The $h_{D,c}$ metric}

The function $h_{D,c}$ was introduced in \cite{dovgoshey2016comparison}.
For all $x,y\in D\subsetneq\mathbb{R}^n$,
\begin{equation*}
h_{D,c}(x,y)=\log\left(1+c\frac{|x-y|}{\sqrt{d(x)d(y)}}\right).
\end{equation*}
When $c\geq 2$, the function is a metric.
Throughout this paper, we assume that the constant $c\geq 2$.

It follows from \cite{dovgoshey2016comparison} that if $c\geq 2$, then for all $x,y\in D\in \{\mathbb{B}^n, \mathbb{H}^n\}$,
\begin{equation*}
\frac{1}{c}h_{D,c}(x,y)\leq \rho_D(x,y) \leq 2h_{D,c}(x,y).
\end{equation*}

\medskip
\subsection{The $t$ metric}

Rainio and Vuorinen introduced the $t$ metric and studied its geometric properties \cite{t}.
Let $D$ be an open set, for all $x,y\in D\subsetneq\mathbb{R}^n$ the $t$ metric is defined as
\begin{equation*}
t_D(x,y)=\frac{|x-y|}{|x-y|+d(x)+d(y)}.
\end{equation*}

\bigskip

\section{The $\tilde{c}$ metric}

\begin{definition}{\rm{\cite{jw7}}}
Let $(X,d)$ be a metric space.
If for all $x, y, z, w\in X$, there holds
\begin{equation}\label{ptolemy}
d(x, y)d(z, w)\leq d(x, z)d(y, w)+d(x, w)d(y, z),
\end{equation}
then we call $(X,d)$ is a Ptolemy space and inequality {\rm{(\ref{ptolemy})}} is the Ptolemy inequality.
\end{definition}

\medskip

\begin{theorem}\label{mainthm}
Let $(X,d)$ be a metric space, and let $D\subsetneq X$ be an open set with $\partial D \ne\emptyset$.
Then for all $x, y\in D$
 \begin{equation}
 \tilde{c}_D(x,y)=\underset{p\in\partial D}{\sup} \frac{d(x,y)}{d(x,p)\vee d(y,p)} \nonumber
  \end{equation}
 is a metric.
\end{theorem}
\begin{proof}
Obviously, we have $\tilde{c}_D(x,y)=\tilde{c}_D(y,x)$, $\tilde{c}_D(x,y)\geq0$,
and $\tilde{c}_D(x,y)=0$ if and only if $x=y$.
It is only necessary to prove the triangle inequality.

\medskip

First, we prove that $\tilde{c}_D(x,y)\leq \tilde{c}_D(x,z)+\tilde{c}_D(y,z)$ holds for any $p\in X$ and $D=X\backslash \{p\}$.
Without loss of generality, we may assume that $d(x,p)\leq d(y,p)$.

\textbf{Case 1.} $d(y,p)\leq d(z,p)$. Then the triangle inequality  is equivalent to
 \begin{equation}
 \dfrac{d(x,y)}{d(y,p)}\leq\dfrac{d(x,z)}{d(z,p)}+\dfrac{d(y,z)}{d(z,p)}. \nonumber
 \end{equation}
By the Ptolemy inequality
  \begin{align}
 d(x,y)d(z,p)&\leq d(x,z)d(y,p)+d(x,p)d(y,z), \nonumber
 \end{align}
we obtain
\begin{align}
 \dfrac{d(x,y)}{d(y,p)}\leq  \dfrac{d(x,z)}{d(z,p)}+ \dfrac{d(x,p)d(y,z)}{d(y,p)d(z,p)}\leq\dfrac{d(x,z)}{d(z,p)}+ \dfrac{d(y,p)d(y,z)}{d(y,p)d(z,p)}
=\dfrac{d(x,z)}{d(z,p)}+\dfrac{d(y,z)}{d(z,p)}.\nonumber
\end{align}


\textbf{Case 2.} $d(x,p)\leq d(z,p)\leq d(y,p)$. Then the triangle inequality is equivalent to
\begin{equation}
 \dfrac{d(x,y)}{d(y,p)}\leq \dfrac{d(x,z)}{d(z,p)}+\dfrac{d(y,z)}{d(y,p)}. \nonumber
 \end{equation}
By the triangle inequality we easily obtain that
\begin{equation}
\dfrac{d(x,y)}{d(y,p)}\leq \dfrac{d(x,z)}{d(y,p)}+\dfrac{d(y,z)}{d(y,p)}\leq \dfrac{d(x,z)}{d(z,p)}+\dfrac{d(y,z)}{d(y,p)}. \nonumber
 \end{equation}


\textbf{Case 3.} $d(z,p)\leq d(x,p)$. Then the triangle inequality is equivalent to
 \begin{equation}
 \dfrac{d(x,y)}{d(y,p)}\leq \dfrac{d(x,z)}{d(x,p)}+\dfrac{d(y,z)}{d(y,p)}. \nonumber
 \end{equation}
By the triangle inequality we easily obtain that
\begin{equation}
\dfrac{d(x,y)}{d(y,p)}\leq \dfrac{d(x,z)}{d(y,p)}+\dfrac{d(y,z)}{d(y,p)}\leq \dfrac{d(x,z)}{d(x,p)}+\dfrac{d(y,z)}{d(y,p)}. \nonumber
 \end{equation}

From \textbf{Case 1-3} we obtain that for any $p\in X$, $x, y, z\in X\backslash \{p\} $, there holds
  \begin{equation*}
  \tilde{c}_{X\backslash \{p\}}(x,y)\leq \tilde{c}_{X\backslash \{p\}}(x,z)+\tilde{c}_{X\backslash \{p\}}(y,z).
  \end{equation*}

\medskip
Next we prove for any open set $D\subsetneq X$ with $\partial D \ne\emptyset$, the triangle inequality still holds.

For any $p\in \partial D$, $x, y, z\in D$, we have $p\in X$, $x, y, z\in X\backslash \{p\}$.
Then
\begin{align}
 \frac{d(x,y)}{d(x,p)\vee d(y,p)}&\leq \frac{d(x,z)}{d(x,p)\vee d(z,p)}+\frac{d(y,z)}{d(y,p)\vee d(z,p)}\nonumber\\
 &\leq
 \underset{p\in\partial D}{\sup}\frac{d(x,z)}{d(x,p)\vee d(z,p)}+\underset{p\in\partial D}{\sup} \frac{d(y,z)}{d(y,p)\vee d(z,p)}.\nonumber
\end{align}
Due to the arbitrariness of $p$, we have
 \begin{equation}
\underset{p\in\partial D}{\sup}\frac{d(x,y)}{d(x,p)\vee d(y,p)}\leq \underset{p\in\partial D}{\sup} \frac{d(x,z)}{d(x,p)\vee d(z,p)}+\underset{p\in\partial D}{\sup} \frac{d(y,z)}{d(y,p)\vee d(z,p)},\nonumber\\
 \end{equation}
namely
 \begin{equation*}
 \tilde{c}_D(x,y)\leq \tilde{c}_D(x,z)+\tilde{c}_D(y,z).
 \end{equation*}

\medskip

In conclusion, $\tilde{c}_D(x,y)$ is a metric.
\end{proof}

\medskip

Since $\mathbb{R}^n$ is a Ptolemy space, the following corollary can be easily obtained.
\begin{corollary}
Let $D\subsetneq \mathbb{R}^n$ be an open set with $\partial D \ne\emptyset$.
Then for all $x,y\in D$,
\begin{equation*}
\tilde{c}_D(x,y)=\underset{p\in\partial D}{\sup} \frac{|x-y|}{|x-p|\vee|y-p|}
\end{equation*}
is a metric.
\end{corollary}

Unless otherwise specified, we always assume that $X=\mathbb{R}^n$, $D$ is an open proper set of $\mathbb{R}^n$, and $\partial D \ne\emptyset$ in the sequel.


We will prove several inequalities that are commonly used in the following context.
\textbf{Lemma \ref{lemmavee}(2)} has already been proven in the proof of Theorem 3.1  in \cite{jw}, but  was not presented as an independent result. For the sake of completeness, we include its proof here.

\medskip
\begin{lemma}\label{lemmavee}
For $q\geq 1$ and all $x,y\in D\subsetneq \mathbb{R}^n$, we have

(1)  $d_{min}\leq\underset{p\in \partial D}\inf(|x-p|\vee|y-p|)\leq |x-y|+d_{min};$

(2)  $d(x)d(y)\leq \underset{p\in \partial D}\inf(|x-p||y-p|) \leq d_{min}(d_{min}+|x-y|);$

(3)  $2^{1/q}d_{min}\leq \underset{p\in \partial D}\inf\left(|x-p|^q+|y-p|^q\right)^{1/q}\leq2^{1/q}(|x-y|+d_{min}).$
\end{lemma}
\begin{proof}

(1) For all $x,y\in D$, $p\in \partial D$, by the triangle inequality we easily obtain that
\begin{equation*}
d_{min}\leq |x-p|\land|y-p|\leq |x-p|\vee|y-p|\leq |x-y|+|x-p|\land|y-p|.
\end{equation*}
Due to the arbitrariness of $p$, we have
\begin{equation}
 d_{min}\leq\underset{p\in \partial D}\inf(|x-p|\vee|y-p|)\leq |x-y|+d_{min}\nonumber.
\end{equation}

(2) For all $x,y\in D$, $p\in \partial D$, by the triangle inequality we easily obtain that
\begin{align*}
|x-p||y-p|&\leq |x-p|(|x-p|+|x-y|),\\
|x-p||y-p|&\leq |y-p|(|y-p|+|x-y|),
\end{align*}
and
\begin{equation*}
|x-p||y-p|\geq d(x)d(y).
\end{equation*}
Due to the arbitrariness of $p$, we have
\begin{equation*}
d(x)d(y)\leq \underset{p\in \partial D}\inf(|x-p||y-p|)
\leq d_{min}(d_{min}+|x-y|).
\end{equation*}

(3) For all $x,y\in D$, $p\in \partial D$, it is easy to see that
\begin{gather}
|x-p|^q+|y-p|^q\leq 2(|x-p|\vee |y-p|)^q\nonumber ,\\
|x-p|^q+|y-p|^q\geq 2(|x-p|\land |y-p|)^q\nonumber.
\end{gather}
Then
\begin{align}
\left(|x-p|^q+|y-p|^q\right)^{1/q}&\leq 2^{1/q}\left(|x-p|\vee|y-p|\right)\leq 2^{1/q}\left(|x-y|+|x-p|\land|y-p|\right),\nonumber\\
\left(|x-p|^q+|y-p|^q\right)^{1/q}&\geq 2^{1/q}\left(|x-p|\land|y-p|\right)\geq 2^{1/q}d_{min}.\nonumber
\end{align}
Due to the arbitrariness of $p$, we have
\begin{equation*}
2^{1/q}d_{min}\leq\underset{p\in \partial D}\inf\left(|x-p|^q+|y-p|^q\right)^{1/q}\leq2^{1/q}(|x-y|+d_{min})\label{b1}\\.
\end{equation*}

\end{proof}

By \textbf{lemma \ref{lemmavee}}, we immediately obtain the following lemma.
\begin{lemma}\label{ctcb}
For all $x,y\in D\subsetneq \mathbb{R}^n$, \\
(1)  $\frac{|x-y|}{|x-y|+d_{min}}\leq \tilde{c}_D(x,y) \leq \frac{|x-y|}{d_{min}};$\\
(2)  $\frac{|x-y|}{d_{min}(d_{min}+|x-y|)}\leq c_D(x,y)\leq \frac{|x-y|}{d(x)d(y)};$\\
(3)  $\frac{|x-y|}{2^{1/q}(|x-y|+d_{min})}\leq b_{D,q}(x,y) \leq \frac{|x-y|}{2^{1/q}d_{min}}.$
\end{lemma}

\bigskip

\section{Inclusion Properties}

In this section, we study inclusion relation between the $\tilde{c}$ metric balls and some related hyperbolic type metric balls.
Let $(D, d)$ be a metric space. The definition of the metric ball $B_d(x,r)$ is as follows:
\begin{equation*}
B_d(x,r)=\{ y\in D:d(x,y)<r\}.
\end{equation*}

\textbf{Theorem \ref{bct}} proves the inclusion relation between the $\tilde{c}$ metric ball and the Barrlund metric ball.
In particular, when $q=1$, the inclusion relation between the $\tilde{c}$ metric ball $B_{\tilde{c}}$ and the triangular ratio metric ball $B_s$ is shown.

\begin{theorem}\label{bct}
For any $x\in D\subsetneq \mathbb{R}^n$ and $ r\in(0,1)$,
\begin{equation}
B_{b_{D,q}}(x,R_1)\subset B_{\tilde{c}}(x,r)\subset B_{b_{D,q}}(x,R_2),\nonumber
\end{equation}
where $R_1=\frac{r}{2^{1/q}(1+r)}$ and $R_2=\frac{r}{2^{1/q}(1-r)}.$ Moreover, $R_2/R_1\to 1$ as $r\to0$.
\end{theorem}
\begin{proof}
First we prove the left-hand side containment relation.
For any $y\in B_{b_{D,q}}(x,R_1)$, by \textbf{Lemma \ref{ctcb}(3)} we obtain that
\begin{equation*}
\frac{|x-y|}{2^{1/q}(|x-y|+d_{min})}\leq b_{D,q}(x,y)<R_1,
\end{equation*}
therefore
\begin{equation}
|x-y|<rd_{min}\nonumber.
\end{equation}
By \textbf{Lemma \ref{ctcb}(1)} we obtain that
\begin{equation}
\tilde{c}_D(x,y)
\leq \frac{|x-y|}{d_{min}}<\frac{rd_{min}}{d_{min}}=r.\nonumber
\end{equation}
Thus, $B_{b_{D,q}}(x,R_1)\subset B_{\tilde{c}}(x,r).$

\medskip

Next, we prove the right-hand side containment relation. For any $y\in B_{\tilde{c}}(x,r)$, by \textbf{Lemma \ref{ctcb}(1)} we obtain that
\begin{equation*}
 \frac{|x-y|}{d_{min}+|x-y|}\leq \tilde{c}_D(x,y)
 <r,\nonumber
\end{equation*}
therefore
\begin{equation}\label{star}
|x-y|<\frac{r}{1-r}d_{min}.
\end{equation}
By \textbf{Lemma \ref{ctcb}(3)} we obtain that
\begin{equation*}
B_{b_{D,q}}(x,y)
\leq\frac{|x-y|}{2^{1/q}d_{min}}<\frac{r}{2^{1/q}(1-r)}.
\end{equation*}
Thus, $B_{\tilde{c}}(x,r)\subset B_{b_{D,q}}(x,R_2). $

\medskip

Clearly,
\begin{equation}
\underset{r\to0}\lim \frac{R_2}{R_1}=1. \nonumber
\end{equation}

\end{proof}

\medskip

We easily obtain the following corollary.
\begin{corollary}
For any $x\in D\subsetneq \mathbb{R}^n$ and $ r\in(0,1)$,
\begin{equation}
B_s(x,R_1)\subset B_{\tilde{c}}(x,r)\subset B_s(x,R_2),\nonumber
\end{equation}
where $R_1=\frac{r}{2(1+r)}$ and $R_2=\frac{r}{2(1-r)}$. Moreover, $R_2/R_1\to 1$ as $r\to0$.
\end{corollary}

\medskip

\textbf{Theorem \ref{cct}} proves the inclusion relation between the $\tilde{c}$ metric ball $B_{\tilde{c}}$ and the Cassinian metric ball $B_c$.
\begin{theorem}\label{cct}
For any $x\in D\subsetneq \mathbb{R}^n$ and $ r\in(0,1)$,
\begin{equation}
B_c(x,R_1)\subset B_{\tilde{c}}(x,r)\subset B_c(x,R_2),\nonumber
\nonumber
\end{equation}
where $R_1=\frac{r}{(1+r)d(x)}$ and $R_2=\frac{r}{(1-r)d(x)}.$ Moreover, $R_2/R_1\to 1$ as $r\to0$.
\end{theorem}
\begin{proof}
First we prove the left-hand side containment relation.
For any $y\in B_c(x,R_1)$, by \textbf{Lemma \ref{ctcb}(2)} we obtain that
\begin{equation*}
\frac{|x-y|}{d_{min}(d_{min}+|x-y|)}\leq c_D(x,y)<R_1=\frac{r}{(1+r)d(x)}\leq \frac{r}{(1+r)d_{min}},
\end{equation*}
therefore
\begin{equation}
|x-y|<rd_{min}\nonumber.
\end{equation}
By \textbf{Lemma \ref{ctcb}(1)} we obtain that
\begin{equation}
\tilde{c}_D(x,y)
\leq \frac{|x-y|}{d_{min}}<\frac{rd_{min}}{d_{min}}=r.\nonumber
\end{equation}
Thus, $B_c(x,R_1)\subset B_{\tilde{c}}(x,r).$

Next, we prove the right-hand side containment relation.
For any $y\in B_{\tilde{c}}(x,r)$, by \textbf{Lemma \ref{ctcb}(1)} we obtain that
\begin{equation}
\frac{|x-y|}{|x-y|+d(y)}\leq\frac{|x-y|}{|x-y|+d_{min}}
\leq \tilde{c}_D(x,y)<r\nonumber,
\end{equation}
therefore
\begin{equation}
|x-y|<\frac{rd(y)}{1-r}\nonumber.
\end{equation}
By \textbf{Lemma \ref{ctcb}(2)} we obtain that
\begin{equation}
c_D(x,y)\leq\frac{|x-y|}{d(x)d(y)}<\frac{r}{(1-r)d(x)}. \nonumber
\end{equation}
Thus, $B_{\tilde{c}}(x,r)\subset B_{{c}}(x,R_2).$

\medskip

Clearly,
\begin{equation}
\underset{r\to0}\lim \frac{R_2}{R_1}=1. \nonumber
\end{equation}

\end{proof}

\medskip

The following theorem proves the inclusion relation between the $\tilde{c}$ metric ball $B_{\tilde{c}}$ and the distance ratio metric ball $B_j$.
\begin{theorem}\label{thm:j}
For any $x\in D\subsetneq \mathbb{R}^n$ and $ r\in(0,1)$,
\begin{equation}
B_j(x,R_1)\subset B_{\tilde{c}}(x,r)\subset B_j(x,R_2),\nonumber
\end{equation}
where $R_1=\log(1+r)$ and $R_2=\log\left(1+\frac{r}{1-r}\right).$ Moreover, $R_2/R_1\to 1$ as $r\to0$.
\end{theorem}
\begin{proof}
First we prove the left-hand side containment relation.
For any $y\in B_j(x,R_1)$, by $j_D(x,y)<R_1$ we obtain that
\begin{equation}
|x-y|<rd_{min}\nonumber.
\end{equation}
By \textbf{Lemma \ref{ctcb}(1)} we obtain that
\begin{equation}
\tilde{c}_D(x,y)
\leq \frac{|x-y|}{d_{min}}<\frac{rd_{min}}{d_{min}}=r.\nonumber
\end{equation}
Thus, $B_j(x,R_1)\subset B_{\tilde{c}}(x,r).$

\medskip

Next, we prove the right-hand side containment relation.
For any $y\in B_{\tilde{c}}(x,r)$, by \textbf{Lemma \ref{ctcb}(1)} and (\ref{star}), we obtain that
\begin{equation*}
j_D(x,y)=\log\left(1+\frac{|x-y|}{d_{min}}\right)<\log\left(1+\frac{r}{1-r}\right).
\end{equation*}
Thus, $B_{\tilde{c}}(x,r)\subset B_j(x,R_2). $

\medskip

Clearly,
\begin{equation}
\underset{r\to0}\lim \frac{R_2}{R_1}=1. \nonumber
\end{equation}

\end{proof}

\medskip

By the well-known inequality (\ref{jrhoB}), it follows that for any  $x\in D\in \{\mathbb{B}^n, \mathbb{H}^n\}$, $r>0$,
\begin{equation}\label{jrho}
B_\rho(x,r)\subset B_j(x,r) \subset B_\rho(x,2r).
\end{equation}
Therefore, by \textbf{Theorem \ref{thm:j}} and inequality (\ref{jrho})
we obtain the inclusion relation between the $\tilde{c}$ metric ball $B_{\tilde{c}}$ and the hyperbolic metric ball $B_{\rho}$.

\medskip

\begin{theorem}
For any $x\in D\in \{\mathbb{B}^n, \mathbb{H}^n\}$ and $ r\in(0,1)$,
\begin{equation*}
B_\rho(x,R_1)\subset B_{\tilde{c}}(x,r)\subset B_\rho(x,R_2),
\end{equation*}
where $R_1=\log(1+r)$ and $R_2=2\log\left(1+\frac{r}{1-r}\right).$ Moreover, $R_2/R_1\to 2$ as $r\to0$.
\end{theorem}
\begin{proof}
First we prove the left-hand side containment relation.
By \textbf{Theorem \ref{thm:j}} and inequality (\ref{jrho}),
\begin{equation*}
B_\rho(x,R_1)\subset B_j(x,R_1)\subset B_{\tilde{c}}(x,{\rm e}^{R_1}-1).
\end{equation*}
Thus, $B_\rho(x,R_1)\subset B_{\tilde{c}}(x,r)$.

\medskip

Next, we prove the right-hand side containment relation.
By \textbf{Theorem \ref{thm:j}} and inequality (\ref{jrho}),
\begin{equation*}
B_{\tilde{c}}(x,r)\subset B_j\left(x,\log\left(1+\frac{r}{1-r}\right)\right)\subset B_\rho\left(x,2\log\left(1+\frac{r}{1-r}\right)\right).
\end{equation*}
Thus, $B_{\tilde{c}}(x,r)\subset B_\rho(x,R_2)$.

\medskip

By l'H\"opital rule, it is easy to see that
\begin{equation*}
\underset{r\to0}\lim \frac{R_2}{R_1}=\underset{r\to0}\lim\frac{2\log\left(1+\frac{r}{1-r}\right)}{\log(1+r)}=2.
\end{equation*}

\end{proof}

\medskip

By \textbf{Theorem \ref{thm:j}} and inequality (\ref{jk})
we obtain the inclusion relation between the $\tilde{c}$ metric ball $B_{\tilde{c}}$ and the quasihyperbolic metric ball $B_k$.
\begin{lemma}\label{lemma:k}{\rm{\cite[Lemma 3.7] {jw20}}}
For any $x\in D\subsetneq \mathbb{R}^n, y\in B(x,d(x))$,
\begin{equation*}
k_D(x,y)\leq \log\left(1+\frac{|x-y|}{d(x)-|x-y|}\right).
\end{equation*}
\end{lemma}

\medskip

\begin{theorem}
For any $x\in D\subsetneq \mathbb{R}^n$ and $ r\in\left(0,\frac{1}{2}\right)$,
\begin{equation}
B_k(x,R_1)\subset B_{\tilde{c}}(x,r)\subset B_k(x,R_2),\nonumber
\end{equation}
where $R_1=\log(1+r)$ and $R_2=\log\left(1+\frac{r}{1-2r}\right).$ Moreover, $R_2/R_1\to 1$ as $r\to0$.
\end{theorem}
\begin{proof}
First we prove the left-hand side containment relation.
For any $y\in B_k(x,R_1)$, by $k_D(x,y)<R_1$ and inequality (\ref{jk}) we obtain that
\begin{equation}
j_D(x,y)\leq k_D(x,y)<R_1.\nonumber
\end{equation}
By \textbf{Theorem \ref{thm:j}} we have $B_k(x,R_1)\subset B_j(x,R_1)\subset B_{\tilde{c}}(x,r)$.

\medskip

Next, we prove the right-hand side containment relation.
For any $y\in B_{\tilde{c}}(x,r)$, by \textbf{Theorem \ref{thm:j}}
we obtain that $y\in B_{\tilde{c}}(x,r)\subset B_j\left(x,\log\left(1+\frac{r}{1-r}\right)\right)$. Then
\begin{equation}
j_D(x,y)<\log\left(1+\frac{r}{1-r}\right)\nonumber.
\end{equation}
Hence
\begin{equation}
|x-y|<\frac{r}{1-r}d_{min}<d_{min}\nonumber.
\end{equation}
Then we have $y\in B(x,d_{min}) \subset B(x,d(x)).$
By \textbf{Lemma \ref{lemma:k}} we obtain that
\begin{equation}
k_D(x,y)\leq \log \left(1+\frac{|x-y|}{d(x)-|x-y|}\right)<\log\left(1+\frac{|x-y|}{d_{min}-|x-y|}\right)<\log\left(1+\frac{r}{1-2r}\right)\nonumber.
\end{equation}
Thus, $B_{\tilde{c}}(x,r)\subset B_k(x,R_2)$.

\medskip

Clearly,
\begin{equation}
\underset{r\to0}\lim \frac{R_2}{R_1}=1. \nonumber
\end{equation}

\end{proof}

\medskip

The following theorem proves the inclusion relation between the $\tilde{c}$ metric ball $B_{\tilde{c}}$ and the $h_{D,c}$ metric ball $B_{h_{D,c}}$.
\begin{theorem}
For any $x\in D\subsetneq \mathbb{R}^n$ and $ r\in(0,1)$,
\begin{equation}
B_{h_{D,c}}(x,R_1)\subset B_{\tilde{c}}(x,r)\subset B_{h_{D,c}}(x,R_2),\nonumber
\end{equation}
where $R_1=\log(1+cr)$ and $R_2=\log\left(1+\frac{cr}{1-r}\right).$ Moreover, $R_2/R_1\to 1$ as $r\to0$.
\end{theorem}
\begin{proof}
First we prove the left-hand side containment relation.
For any $y\in B_{h_{D,c}}(x,R_1)$, by $h_{D,c}(x,y)<R_1$ we obtain that
\begin{equation}
|x-y|<r\sqrt{d(x)d(y)}<rd_{max}\nonumber.
\end{equation}
By the definition of $\tilde{c}$ metric, we have
\begin{equation}
\tilde{c}_D(x,y)=\frac{|x-y|}{\underset{p\in \partial D}\inf(|x-p|\vee |y-p|)}\leq \frac{|x-y|}{d_{max}}<\frac{rd_{max}}{d_{max}}=r.\nonumber
\end{equation}
Thus, $B_{h_{D,c}}(x,R_1)\subset B_{\tilde{c}}(x,r)$.

\medskip

Next, we prove the right-hand side containment relation.
For any $y\in B_{\tilde{c}}(x,r)$, by \textbf{Lemma \ref{ctcb}(1)} and (\ref{star})
we obtain that
\begin{equation*}
h_{D,c}(x,y)=\log\left(1+c\frac{|x-y|}{\sqrt{d(x)d(y)}}\right)\leq\log\left(1+c\frac{|x-y|}{d_{min}}\right)<\log\left(1+c\frac{r}{1-r}\right).
\end{equation*}
Thus, $B_{\tilde{c}}(x,r)\subset B_{h_{D,c}}(x,R_2)$.

\medskip

Clearly,
\begin{equation}
\underset{r\to0}\lim \frac{R_2}{R_1}=1. \nonumber
\end{equation}

\end{proof}

\medskip

\textbf{Theorem \ref{ctt}} proves the inclusion relation between the $\tilde{c}$ metric ball $B_{\tilde{c}}$ and the $t$ metric ball $B_t$.
\begin{theorem}\label{ctt}
For any $x\in D\subsetneq \mathbb{R}^n$ and $ r\in(0,1)$,
\begin{equation}
B_t(x,R_1)\subset B_{\tilde{c}}(x,r)\subset B_t(x,R_2),\nonumber
\end{equation}
where $R_1=\frac{r}{2+r}$ and $R_2=\frac{r}{2(1-r)}.$ Moreover, $R_2/R_1\to 1$ as $r\to0$.
\end{theorem}
\begin{proof}
First we prove the left-hand side containment relation.
For any $y\in B_t(x,R_1)$ we have
\begin{equation}
\frac{|x-y|}{|x-y|+2d_{max}}\leq t(x,y)< R_1,\nonumber
\end{equation}
therefore
\begin{equation*}
|x-y|<rd_{max}.
\end{equation*}
By the definition of $\tilde{c}$ metric, we obtain
\begin{equation}
\tilde{c}_D(x,y)=\frac{|x-y|}{\underset{p\in \partial D}\inf(|x-p|\vee |y-p|)}\leq \frac{|x-y|}{d_{max}}<\frac{rd_{max}}{d_{max}}=r.\nonumber
\end{equation}
Thus, $B_t(x,R_1)\subset B_{\tilde{c}}(x,r)$.

\medskip

Next, we prove the right-hand side containment relation.
For any $y\in B_{\tilde{c}}(x,r)$, by \textbf{Lemma \ref{ctcb}(1)} and (\ref{star})
we obtain that
\begin{equation*}
t(x,y)=\frac{|x-y|}{|x-y|+d(x)+d(y)}\leq \frac{|x-y|}{d(x)+d(y)} \leq\frac{|x-y|}{2d_{min}}<\frac{r}{2(1-r)}.
\end{equation*}
Thus, $B_{\tilde{c}}(x,r)\subset B_t(x,R_2). $

\medskip

Clearly,
\begin{equation}
\underset{r\to0}\lim \frac{R_2}{R_1}=1. \nonumber
\end{equation}

\end{proof}

\bigskip

\section{Distortion Property of M\"obius Transformations}

In this section, we consider the distortion properties of M\"obius transformations on the unit ball under the $\tilde{c}$ metric.

If $D$ is a subset of $\mathbb{R}^n$, then the group of all M\"obius transformations that map $D$ onto itself is denoted by $\mathscr{G}\kern -0.25em\mathscr{M}(D)$.
For $a\in \mathbb{R}^n\backslash \{0\}$, let $a^*=\frac{a}{{|a|}^2}, 0^*=\infty$ and $\infty^*=0$. Let
\begin{equation}\label{eq}
\sigma_a(x)=a^*+r^2{(x-a^*)}^*,\quad r^2={|a^*|}^2-1
\end{equation}
be the inversion in the sphere $S^{n-1}(a^*,r)$\,.

\medskip

\begin{lemma}{\rm{\cite[Theorem 4.4.8] {ratcliffe1994foundation}}}\label{distortion}
Let $f\in \mathscr{G}\kern -0.25em\mathscr{M}(\mathbb{B}^n)$.
Then $f$ is an orthogonal transformation that maps $\mathbb{R}^n$ to $\mathbb{R}^n$ if and only if $f(0)=0$.
\end{lemma}

\medskip

\begin{theorem}
Let $f:\mathbb{B}^n\to\mathbb{B}^n$ be a M\"obius transformation, and $a\in \mathbb{B}^n$ with $f(a)=0$.
Then, for all $x,y\in \mathbb{B}^n$,
  \begin{equation}
    \frac{1-|a|}{1+|a|} \tilde{c}_{\mathbb{B}^n}(x,y) \leq \tilde{c}_{\mathbb{B}^n}(f(x),f(y))\leq \frac{1+|a|}{1-|a|} \tilde{c}_{\mathbb{B}^n}(x,y). \nonumber
  \end{equation}
\end{theorem}
\begin{proof}
If $a=0$, according to \textbf{Lemma  \ref{distortion}} we know that $f$ is an orthogonal transformation and preserves the $\tilde{c}$ metric, i.e.,
for all $x,y \in \mathbb{B}^n$,
\begin{equation*}
\tilde{c}_{\mathbb{B}^n}(f(x),f(y))=\tilde{c}_{\mathbb{B}^n}(x,y),
\end{equation*}
thus the conclusion is obviously true.

If $a\neq0$, since $f\circ \sigma_a\in \mathscr{G}\kern -0.25em\mathscr{M}(\mathbb{B}^n)$ and $f(\sigma_a(0))=f(a)=0$,
by \textbf{Lemma \ref{distortion}}, we have $\psi_1\equiv f\circ \sigma_a$ is an orthogonal transformation.
Hence $f=\psi_1\circ {\sigma_a}^{-1}=\psi_1\circ \sigma_a$. Therefore for all $x, y \in \mathbb{B}^n$,
\begin{equation*}
\tilde{c}_{\mathbb{B}^n}(f(x),f(y))=\tilde{c}_{\mathbb{B}^n}(\sigma_a(x),\sigma_a(y)).
\end{equation*}
By the equality (\ref{eq}) we have
\begin{equation*}
|\sigma_a(x)-\sigma_a(y)|=\frac{r^2|x-y|}{|x-a^*||y-a^*|}.
\end{equation*}

For any $p\in \partial \mathbb{B}^n$, let
\begin{equation*}
Q(x,y,p)\equiv\left(\frac{|\sigma_a(x)-\sigma_a(y)|}{|\sigma_a(x)-\sigma_a(p)|\vee|\sigma_a(y)-\sigma_a(p)|}\right)\bigg/\left(\frac{|x-y|}{|x-p|\vee|y-p|}\right).
\end{equation*}

\medskip

Without loss of generality, we may assume that $|x-p|\leq|y-p|$.

First we prove the right-hand side inequality.

\textbf{Case 1.} $|\sigma_a(x)-\sigma_a(p)|\leq|\sigma_a(y)-\sigma_a(p)|$. Then
\begin{align}
Q(x,y,p)=\left(\frac{|\sigma_a(x)-\sigma_a(y)|}{|\sigma_a(y)-\sigma_a(p)|}\right)\bigg/\left(\frac{|x-y|}{|y-p|}\right)=\frac{|p-a^*|}{|x-a^*|}\nonumber.
\end{align}
By $|p-a^*|\leq1+\frac{1}{|a|}$ and $|x-a^*|\geq\frac{1}{|a|}-1$, we have
\begin{equation}
Q(x,y,p)\leq\frac{1+\frac{1}{|a|}}{\frac{1}{|a|}-1}=\frac{1+|a|}{1-|a|}. \nonumber
\end{equation}

\textbf{Case 2.} $|\sigma_a(x)-\sigma_a(p)|\geq|\sigma_a(y)-\sigma_a(p)|$, i.e., $\frac{|x-p|}{|x-a^*|}\geq\frac{|y-p|}{|y-a^*|}$. Then
\begin{align}
Q(x,y,p)=\left(\frac{|\sigma_a(x)-\sigma_a(y)|}{|\sigma_a(x)-\sigma_a(p)|}\right)\bigg/\left( \frac{|x-y|}{|y-p|}\right)&=\frac{|y-p||p-a^*|}{|y-a^*||x-p|}\nonumber\\
&\leq\frac{|x-p|}{|x-a^*|}\frac{|p-a^*|}{|x-p|}\nonumber\\
&=\frac{|p-a^*|}{|x-a^*|}\nonumber.
\end{align}
Moreover,
\begin{equation}
Q(x,y,p)\leq\frac{1+\frac{1}{|a|}}{\frac{1}{|a|}-1}=\frac{1+|a|}{1-|a|}. \nonumber
\end{equation}

Thus for $x,y\in\mathbb{B}^n$ and $p\in \partial \mathbb{B}^n$, we have
\begin{align}
\frac{|\sigma_a(x)-\sigma_a(y)|}{|\sigma_a(x)-\sigma_a(p)|\vee|\sigma_a(y)-\sigma_a(p)|}&\leq \frac{1+|a|}{1-|a|}\frac{|x-y|}{|x-p|\vee|y-p|}\nonumber\\
&\leq\frac{1+|a|}{1-|a|} \tilde{c}_{\mathbb{B}^n}(x,y).\nonumber
\end{align}
Due to the arbitrariness of $p$, we have
\begin{align}
\tilde{c}_{\mathbb{B}^n}(f(x),f(y))=\tilde{c}_{\mathbb{B}^n}(\sigma(x),\sigma(y))&=\frac{|\sigma_a(x)-\sigma_a(y)|}{\underset{p\in\partial \mathbb{B}^n}\inf\left(|\sigma_a(x)-\sigma_a(p)|\vee|\sigma_a(y)-\sigma_a(p)|\right)}\nonumber\\
&\leq\frac{1+|a|}{1-|a|} \tilde{c}_{\mathbb{B}^n}(x,y).\nonumber
\end{align}

\medskip

Next we prove the left-hand side inequality.

Since $\sigma_a\circ f \in  \mathscr{G}\kern -0.25em\mathscr{M}(\mathbb{B}^n)$ and $\sigma_a(f^{-1}0))=\sigma_a(a)=0$, by \textbf{Lemma \ref{distortion}}, we have $\psi_2\equiv \sigma_a\circ f^{-1}$ is an orthogonal transformation. Hence $f^{-1}= {\sigma_a}^{-1}\circ \psi_2=\sigma_a\circ \psi_2$. By the proof of the right-hand side inequality, for all $x,y\in\mathbb{B}^n$, we have
\begin{align*}
\tilde{c}_{\mathbb{B}^n}(f^{-1}(x),f^{-1}(y))&=\tilde{c}_{\mathbb{B}^n}(\sigma_a(\psi_2(x),\sigma_a(\psi_2(y)))\\
&\leq \frac{1+|a|}{1-|a|} \tilde{c}_{\mathbb{B}^n}(\psi_2(x),\psi_2(y))\\
&=\frac{1+|a|}{1-|a|} \tilde{c}_{\mathbb{B}^n}(x,y).
\end{align*}
Therefore
\begin{equation*}
\tilde{c}_{\mathbb{B}^n}(f(x),f(y))\geq \frac{1-|a|}{1+|a|} \tilde{c}_{\mathbb{B}^n}(x,y).
\end{equation*}

\medskip

Thus the proof is complete.
\end{proof}

\bigskip

\section{Quasiconformality of a Bilipschitz Mapping in $\tilde{c}$ Metric}

Bilipschitz mappings with respect to the triangular ratio metric have been studied in \cite[Theorem 4.4] {jw10}.
Due to the similarity between the definitions of triangular ratio metric and $\tilde{c}$ metric,
in this section we will investigate the quasiconformality of bilipschitz mappings in $\tilde{c}$ metric.

\medskip

\begin{theorem}
Let $D\subsetneq \mathbb{R}^n$ be an open set and let $f:D\to fD\subset \mathbb{R}^n$ be a sense-preserving homeomorphism, satisfying L-bilipschitz condition with respect to the $\tilde{c}$ metric, i.e.,
\begin{equation}\label{lip}
{\tilde{c}_D(x,y)}/{L}\leq\tilde{c}_{fD}(f(x),f(y))\leq L\tilde{c}_D(x,y)
\end{equation}
for all $x,y\in D$. Then $f$ is  quasiconformal with the linear dilatation $H(f)\leq L^2$.
\end{theorem}
\begin{proof}
Let $x,y\in D$ with $|x-y|< \min\{d(x), d(y)\}\equiv m_D(x,y)$.
Without loss of generality, we may assume that $d_{min}=d(x)$ and $d(x)=|x-w|, \, w\in\partial D$.
Therefore by \textbf{Lemma \ref{ctcb}(1)} we have
\begin{align}
m_D(x,y)-|x-y|\leq\underset{p\in\partial D}\inf(|x-p|\vee|y-p|) &\leq |x-y|+m_D(x,y). \nonumber
\end{align}
Hence
\begin{equation}\label{cfanwei}
\frac{|x-y|}{m_D(x,y)+|x-y|}\leq \tilde{c}_D(x,y) \leq \frac{|x-y|}{m_D(x,y)-|x-y|}.
\end{equation}
Since
\begin{equation*}
\frac{1}{\tilde{c}_D(x,y)}=\frac{\underset{p\in\partial D}\inf(|x-p|\vee|y-p|)}{|x-y|}\geq \frac{d(x)\vee d(y)}{|x-y|}\geq \frac{m_D(x,y)}{|x-y|}> 1,
\end{equation*}
we  deduce from inequality (\ref{cfanwei}) that
\begin{equation}\label{xyfanwei}
\frac{m_D(x,y)}{\frac{1}{\tilde{c}_D(x,y)}+1}\leq |x-y| \leq \frac{m_D(x,y)}{\frac{1}{\tilde{c}_D(x,y)}-1}.
\end{equation}

For any $z\in D$, let $x,y\in D$ such that $|x-z|=|y-z|=r$, $r$ is small enough to make the following discussion meaningful.
Let
\begin{equation*}
A(x,y,z)=\frac{m_{fD}(f(x),f(z))}{m_{fD}(f(y),f(z))},
\end{equation*}
and $A(x,y,z)\to 1$ as $x, y\to z$.

As $r\to 0$, according to equations (\ref{lip})-(\ref{xyfanwei}), we obtain

\begin{align}
\frac{|f(x)-f(z)|}{|f(y)-f(z)|}&\leq A(x,y,z) \frac{\frac{1}{\tilde{c}_{fD}(f(y),f(z))}+1}{\frac{1}{\tilde{c}_{fD}(f(x),f(z))}-1}\nonumber\\
&\leq A(x,y,z) \frac{\frac{L}{\tilde{c}_{D}(y,z)}+1}{\frac{1}{L\tilde{c}_{D}(x,z)}-1}\nonumber\\
&\leq A(x,y,z) \frac{\frac{L(m_D(y,z)+|y-z|)}{|y-z|}+1}{\frac{m_D(x,z)-|x-z|}{L|x-z|}-1}\nonumber\\
&=A(x,y,z) \frac{|x-z|}{|y-z|} \frac{L^2m_D(y,z)+(L^2+L)|y-z|}{m_D(x,z)-(L+1)|x-z|}\nonumber\\
&\to L^2. \nonumber
\end{align}
Then,
\begin{equation}
H(f(z))=\underset{|x-z|=|y-z|=r\to0^+} \limsup\frac{|f(x)-f(z)|}{|f(y)-f(z)|}\leq L^2.\nonumber
\end{equation}

\medskip

Thus the proof is complete.
\end{proof}

\medskip

{\bf Acknowledgements.}
This research was supported by National Natural Science Foundation of China
(NNSFC) under Grant Nos. 11771400 and 11601485.

\medskip



\end{document}